\documentclass[11pt,english]{article}

\usepackage[T1]{fontenc}
\usepackage[latin9]{inputenc}
\usepackage{geometry}
\usepackage{babel}
\usepackage{epsfig}
\usepackage{amsmath}
\usepackage{amsthm}
\usepackage{amssymb}
\usepackage{graphicx}
\usepackage{booktabs}
\usepackage[numbers]{natbib}
\usepackage[font=small,labelfont=bf]{caption}
\usepackage{varwidth}
\usepackage{caption}
\usepackage{tabularx}
\usepackage[export]{adjustbox}
\usepackage{comment}
\usepackage{multirow}
\usepackage[unicode=true,pdfusetitle,bookmarks=true,bookmarksnumbered=false,bookmarksopen=false,
 breaklinks=false,pdfborder={0 0 1},backref=false,colorlinks=false]{hyperref}

\DeclareMathOperator{\E}{{\mathbb E}}

\DeclareMathOperator{\Z}{{\mathbb Z}}
\DeclareMathOperator{\N}{{\mathbb N}}

\DeclareMathOperator{\R}{{\mathbb R}}

\DeclareMathOperator{\dom}{dom}

\providecommand{\epsilon}{\varepsilon}
\renewcommand{\phi}{\varphi}

\renewcommand{\theta}{\vartheta}

\providecommand{\abs}[1]{\lvert #1 \rvert}
\providecommand{\norm}[1]{\lVert #1 \rVert}
\providecommand{\bnorm}[1]{{\Bigl\lVert #1 \Bigr\rVert}}

\renewcommand{\le}{\leqslant}
\renewcommand{\leq}{\le}
\renewcommand{\ge}{\geqslant}
\renewcommand{\geq}{\ge}

\newtheorem{theorem}{Theorem}
\newtheorem{proposition}[theorem]{Proposition}

\newtheorem{corollary}[theorem]{Corollary}
\theoremstyle{definition}

\newtheorem{remark}[theorem]{Remark}

\usepackage{enumitem}
\setlist[enumerate]{leftmargin=.5in}
\setlist[itemize]{leftmargin=.5in}

\title{Parameter Estimation in an SPDE Model for Cell Repolarisation\thanks{We are grateful to two anonymous referees for very helpful comments and questions.  This research has been partially funded by Deutsche Forschungsgemeinschaft (DFG) - SFB1294/1 - 318763901. RA gratefully acknowledges support by the European Research Council, ERC grant agreement 647812 (UQMSI).}}

\author{{Randolf Altmeyer}\thanks{University of Cambridge, Department of Pure Mathematics \& Mathematical Statistics, Wilberforce Road, CB3 0WB Cambridge, United Kingdom. Email: ra591@cam.ac.uk.} 
\hspace{0.5cm} {Till Bretschneider}\thanks{Department of Computer Science, University of Warwick, Academic Loop Road, CV4 7AL Coventry, United Kingdom. Email: T.Bretschneider@warwick.ac.uk.}
\hspace{0.5cm} {Josef Janák}\thanks{Universität Potsdam, Institut für Mathematik, Karl-Liebknecht-Str. 24/25, 14476 Potsdam, Germany. Email: josefjanak@seznam.cz.}
\hspace{0.5cm} {Markus Reiß}\thanks{Humboldt-Universität zu Berlin, Institut für Mathematik, Unter den Linden 6, 10099 Berlin, Germany. Email: mreiss@math.hu-berlin.de.}
}

\begin{document}

\global\long\def\P{\mathbb{P}}
\global\long\def\dom{\operatorname{dom}}%
\global\long\def\b#1{\mathbb{#1}}%
\global\long\def\c#1{\mathcal{#1}}%
\global\long\def\s#1{{\scriptstyle #1}}%
\global\long\def\u#1#2{\underset{#2}{\underbrace{#1}}}%
\global\long\def\r#1{\xrightarrow{#1}}%
\global\long\def\mr#1{\mathrel{\raisebox{-2pt}{\ensuremath{\xrightarrow{#1}}}}}
\global\long\def\t#1{\left.#1\right|}%
\global\long\def\l#1{\left.#1\right|}%
\global\long\def\f#1{\lfloor#1\rfloor}%
\global\long\def\sc#1#2{\langle#1,#2\rangle}%
\global\long\def\abs#1{\lvert#1\rvert}%
\global\long\def\bnorm#1{\Bigl\lVert#1\Bigr\rVert}%
\global\long\def\wraum{(\Omega,\c F,\P)}%
\global\long\def\fwraum{(\Omega,\c F,\P,(\c F_{t}))}%
\global\long\def\norm#1{\lVert#1\rVert}%

\maketitle

\begin{abstract}
As a concrete setting where stochastic partial differential equations (SPDEs) are able to model real phenomena, we propose a stochastic Meinhardt model for cell repolarisation and study how parameter estimation techniques developed for simple linear SPDE models apply in this situation.
We establish the existence of mild SPDE solutions and we investigate the  impact of the driving noise process on pattern formation in the solution. We then pursue estimation of the diffusion term and show asymptotic normality for our estimator as the space resolution becomes finer.  The finite sample performance is investigated for synthetic and real data.
\end{abstract}

\section{Introduction}

Stochastic partial differential equations (SPDEs) generalize deterministic partial differential equations (PDEs) by introducing driving noise processes into the dynamics. These noise processes encapsulate unresolved and often unknown processes happening at faster scales and random external forces acting on the system. Not only the theory of SPDEs, but also the statistics for SPDEs have recently seen a significant development, paving the way for a realistic modeling of complex phenomena. We demonstrate the ability of SPDE models to describe cell repolarisation patterns and we show how parameter estimation techniques, developed for simplified linear models, apply in more complex and physically relevant situations. We see this as an important step to make theoretical tools also available for concrete experimental setups. For the sake of clarity we focus on a specific stochastic cell polarisation problem, but the methodology has a much broader scope.

The SPDE we are interested in belongs to a general class of activator-inhibitor models, which can be described by two coupled stochastic reaction-diffusion equations of the form
\begin{align}
\begin{cases}
\frac{\partial}{\partial t}A(t,x)=D_{A}\frac{\partial^{2}}{\partial x^{2}}A(t,x)+f_A(X(t,x),x)+\sigma_{A}\xi_{A}(t,x),\\
\frac{\partial}{\partial t}I(t,x)=D_{I}\frac{\partial^{2}}{\partial x^{2}}I(t,x)+f_I(X(t,x),x)+\sigma_{I}\xi_{I}(t,x),
\end{cases}\label{eq:stochasticMeinhardt}
\end{align}
with $X=(A,I)$,  nonlinear functions $f_A$, $f_I$ and with space-time white noise processes $\xi_{A}$, $\xi_{I}$. In cell dynamics, we think of $A$ as a hypothetical signalling molecule that in response to an external signal gradient becomes enriched on one side of the cell, yielding polarity. $I$ counter-acts $A$ so that removal of the signal results in loss of polarity. We here specifically consider repolarisation, where the extracellular signal gradient is inverted so that $A$ is removed from one side of the cell and reappears on the opposite side, cf. Figure \ref{fig:1} below.

In directed animal cell motion cells respond to for example chemical or mechanical extracellular signal gradients by adopting a functional asymmetry in form of a front-rear pattern. Protrusion of the cell front is driven by local oriented growth of a dense network of cytoskeletal actin filaments pushing the cellular envelope \citep{mullins1998interaction}. Myosin-II motor molecules contracting the looser, ubiquitous cortical actin network lining the cell membrane result in retraction of the cell rear in stringent environments \citep{cramer2013mechanism}.

Models for spontaneous symmetry breaking in non-linear reaction diffusion systems by Turing \citep{Turing1952} have been instrumental in understanding biological pattern formation, often paraphrased in form of simple deterministic two-variable activator-inhibitor models such as \eqref{eq:stochasticMeinhardt} without noise terms. Suprathreshold random perturbations can result in fast autocatalytic local growth of the activator variable $A$, which eventually is kept in check by the slower inhibitor $I$. Faster diffusion of the inhibitor compared to the activator prevents formation of nearby activator peaks.

Meinhardt \citep{Meinhardt1999} has been the first to apply such models to cell polarisation in the context of cell migration, where the ratio of activator-inhibitor diffusion can be tuned to either obtain a single stable cell front (Figure \ref{fig:1}(center)), or multiple independent fronts associated with non-directed random cell motility. Various mathematical models for cell polarisation and gradient sensing have been postulated (\citep{Levchenko2002}, \citep{Otsuji2007}, \citep{Jilkine2011}) aiming to capture different aspects of cellular physiology for example with regards to adaptation to extracellular signals, reviewed in \citep{cheng2020roles}. \citep{Lockley_2015} fitted deterministic versions of three different models for cell polarisation to experimental data of cells in a microfluidic chamber responding to inversion of gradients of hydrodynamic shear flow of different strengths \citep{dalous2008reversal}.  The parameter calibration was based on a least-squares approach, implicitly assuming that the deterministic dynamics are corrupted by Gaussian measurement  noise.

Recognising that in confined spaces and with limited number of molecules noise becomes increasingly important, more recently stochastic reaction-diffusion models for different biological problems have been employed, e.g. in \citep{Alonso_2018}, \citep{spill2015hybrid}. Spontaneous symmetry breaking in Turing-type models requires initial random perturbations, but we expect dynamic noise to destabilise patterns if the power of the noise is too large.

\begin{figure}
	\includegraphics[width=\textwidth]{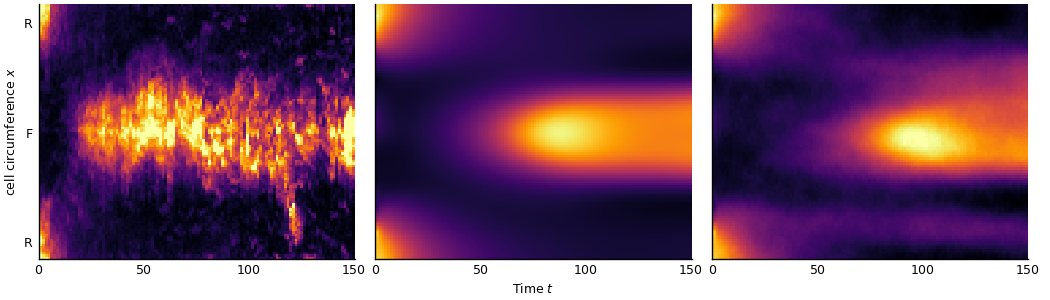}
	\caption{Heat maps for the space-time evolution of the activator $A$, brighter colors mean higher values, space region R denotes the old front/new rear, region F is the new front/old rear; (left) experimental data for measured fluorescence values averaged over several \emph{Dictyostelium} cells reacting to a gradient of shear flow; (center) solution to the deterministic Meinhardt model with experimentally fitted parameters; (right) a typical realisation of the stochastic Meinhardt model with noise level $0.02$.}
\label{fig:1}
\end{figure}

In the current paper we present a stochastic version of the modified Meinhardt two-variable model \citep{Lockley_2015}. We believe that the stochasticity in the data is to some considerable extent due to dynamic noise entering the dynamics as in \eqref{eq:stochasticMeinhardt}. The data generated from such a stochastic Meinhardt model is qualitatively of a different nature compared to a deterministic PDE model corrupted by measurement errors.
We perform a systematic study of the effect of different noise levels on the repolarisation of cells.  One result is that inclusion of moderate levels of noise in the model speeds up the repolarisation of cells, which biologically is interesting because it might be against our intuition that noise would negatively interfere with the formation of a pattern.

Recently, new tools for parameter estimation of stochastic differential equations have been developed, see \citep{Cialenco2018} for an overview. Most approaches focus on estimating coefficients for the linear part of the equation, either from discrete \citep{Cialenco2019a}, \citep{Hildebrandt2019} or spectral \citep{Huebner1995}, \citep{Pasemann} observations, but also aspects of the driving noise have been analysed \citep{Chong}, \citep{Bibinger_2020}. Owing to the physical restriction of being able to measure only local averages, \citep{Altmeyer2019b} have introduced {local measurements} and constructed estimators in a linear SPDE for the diffusion term which are provably rate-optimal. Even more, the proposed estimators apply in a nonparametric setting of spatially varying diffusion and are robust to misspecification of the noise or when subject to certain nonlinearities \citep{Altmeyer2020a}.

We extend the estimation method in \citep{Altmeyer2019b} to cope also with multiple spatial measurements, systems of SPDEs as \eqref{eq:stochasticMeinhardt} and with more general boundary conditions (here periodic boundary conditions will apply), cf. Remark \ref{rem:generalProof}. We shall perform parameter estimation in the stochastic Meinhardt model for cell repolarisation and provide confidence intervals to quantify the uncertainty. In particular, we are interested in determining the diffusion constant for the activator in the Meinhardt model. Although the activator variable in the model cannot be directly related to a specific molecular component, putting limits on how fast the activator spreads can ideally help narrowing down possible mechanisms. For example, spreading of the activator could be down to lateral growth of the actin network (slow), diffusion of chemoattractant receptors within the cell membrane (medium) or diffusion of phospholipid signalling molecules (PIP3) within the cell membrane (fast).

Mathematically, we derive a central limit theorem for our estimator by using advanced tools from stochastic analysis and semigroup theory. We are aware of only one related work \citep{Pasemann2020}, which uses the spectral method to fit parameters of a 2D Fitz-Hugh-Nagumo model for travelling actin waves through cells. For the stochastic Meinhardt model we compare in Section \ref{sec:numerics} below our method with the spectral estimation method.

In the next Section \ref{sec:meinhardt} the stochastic Meinhardt model is introduced, along with a rigorous result on the existence of a solution. Section \ref{sec:effectofnoise} presents main insights how adding noise to the Meinhardt model affects the dynamics and repolarisation. In Section \ref{sec:estimation} estimators for local measurements of the activator are analysed mathematically and applied in Section \ref{sec:numerics} to synthetic data, and in Section \ref{sec:realData} to experimental data. Section \ref{sec:discussion} discusses the main results. All technical details and proofs are deferred to Appendix \ref{subsec:Existence-of-a}. A description of the setup for numerical experiments and real data can be found in Appendix \ref{sec:setup}.

\section{The stochastic Meinhardt model}\label{sec:meinhardt}

Diffusion is considered to take place along the cell contour, and so we study the equation \eqref{eq:stochasticMeinhardt} with $0\leq t\leq  T$ for a fixed time horizon $T>0$ on a circle or 1D-torus $\Lambda=\R/(L\Z)$ of length $L>0$ or, equivalently, on $\Lambda=[0,L]$ with periodic boundary conditions. $A$ describes a membrane-bound autocatalytic activator requiring $f_A$ to be nonlinear, and with diffusion coefficient $D_A$. The production of $A$ is counteracted by a small cytosolic inhibitor $I$ with faster diffusion (that is, $D_I > D_A$), where $f_I$ is linear or nonlinear. In case of the two-variable Meinhardt model the functions $f_A$ and $f_I$ are given by
\begin{align}
f_A(y,x) & =r_{A}\frac{\zeta(x)\left(b_{A}+y_{1}^{2}\right)}{(\zeta_{I}+|y_2|) \left(1+\zeta_{A}y_{1}^{2}\right)}-r_{A}y_1,\quad f_I(y,x) = b_{I}y_1-r_{I}y_2,\label{eq:nonlinearities}
\end{align}
for $y\in\R^{2}$ and $x\in\Lambda$. The function
\begin{equation}
\zeta(x)=1- a \cdot \cos(2\pi x/L)\label{eq:signal}
\end{equation}
corresponds to an extracellular signal, for example a gradient of chemoattractant, which stimulates the production of $A$ with signal strength modulated by a constant $a$.  The extracellular signal is maximal at the center $L/2$ of the front.  The constants $r_{A}$, $r_{I}$ and $b_{A}$, $b_{I}$ are degradation and production rates, $\zeta_{A}$ controls the saturation and the Michaelis-Menten constant $\zeta_{I}$ prevents $f_A$ from exploding. While $b_I$ is fixed in our setup, it will generally depend on the pressure of the signal $\zeta$ \citep{Lockley_2015}. For a more detailed description of the nonlinearities $f_A$, $f_I$ and a stability analysis for varying parameters see \citep{lockley2017image}, \citep{meinhardt2009}.

Additional external forces acting on the cell membrane are modeled by two independent  space-time white noise processes $\xi_{A}$, $\xi_{I} $. By space-time white noise we mean a centered Gaussian process $\xi$ on $[0,T]\times\Lambda$ with covariance function
\[
	\text{Cov}\left(\xi(t,x),\xi(t',x')\right) = \delta(t-t')\delta(x-x').
\]
By integrating formally against test functions, $\xi$ induces an isonormal Gaussian process on $L^2([0,T]\times \Lambda)$. In this way, space-time white noise corresponds to a random Schwartz distribution on $L^2(\Lambda)$ with values in negative Sobolev spaces \cite{Hairer2009}. Since the nonlinearity $f_A(X(t,x),x)$ is not well-defined for a distribution valued process $X$, this means we cannot obtain classical solutions to the SPDE \eqref{eq:stochasticMeinhardt}. After formally integrating the noise, however, $W(t)=\int_0^t \xi(s,\cdot)ds$ is a (cylindrical) Wiener process with values in $L^2(\Lambda)$  \citep{DaPrato2014}, and we can use
the well-developed theory for SPDEs to show that \eqref{eq:stochasticMeinhardt} is well-posed in the mild sense.  The solution even has some minimal spatial regularity measured in the spaces $C^{k+s}(\Lambda;\R^2)$, $k\in\N_0$, $0\leq s<1$, equipped with the norm
\begin{align*}
	\norm{f}_{C^{k+s}(\Lambda;\R^2)} = \norm{f}_{C^k(\Lambda;\R^2)} + \sup_{x\neq y\in\Lambda} \frac{\norm{\frac{d^k}{dx^k}f(x)-\frac{d^k}{dx^k}f(y)}_{\R^2}}{|x-y|^s}.
\end{align*}

\begin{theorem}
\label{thm:existenceUniqueness}Consider the stochastic Meinhardt model corresponding to the SPDE in \eqref{eq:stochasticMeinhardt} with nonlinearities $f_A$, $f_I$ from \eqref{eq:nonlinearities} on $\Lambda=\R/(L\Z)$. Assume for the initial value $(A_{0},I_{0})\in C^{2+s}(\Lambda;\R^{2})$ for $0<s<1/2$. Then there exists a unique mild solution $X=(A,I)\in C([0,T];C^{s}(\Lambda;\R^{2}))$. The solution can be decomposed as $X=\bar{X}+\tilde{X}$, where $\bar{X}$ solves the linear equation \eqref{eq:stochasticMeinhardt} with $f_A=f_I=0$ and zero initial value, and with a perturbation process $\tilde{X}\in C([0,T];C^{2+s}(\Lambda;\R^2))$.
\end{theorem}

While the existence and uniqueness of the mild solution follows from a standard result in \citep{Hairer2009},  the crucial insight of the theorem is the higher regularity of $\tilde{X}$ compared to $X$, which we will exploit in the statistical analysis below to show that the influence of the nonlinearity $f_A$ on the estimation of $D_A$ from data is negligible, while it clearly impacts the nonlinear dynamics of $X$.

For the proof of Theorem \ref{thm:estimate_DA} in Appendix \ref{sec:mild solution} we shall employ the language of stochastic analysis, while for modeling purposes we prefer  \eqref{eq:stochasticMeinhardt} with the physical white noise interpretation.
A more realistic model might consider multiplicative noise levels $\sigma_A$, $\sigma_I$ depending on $X$. Moreover,  also the parameters and initial conditions could be subject to noise \citep{Jilkine2011}. Here, we refrain from this generality and focus on the impact of simple additive space-time white noise in \eqref{eq:stochasticMeinhardt}. Note that neither the model proposed by Meinhardt nor other models suggested in the literature for cell repolarisation include dynamic noise so far.

\section{The effect of noise} \label{sec:effectofnoise}

In Turing-type models for pattern formation noise in the initial condition is required to leave a homogeneous steady state. Its strength determines how fast a suprathreshold level for the activator $A$ starts growing into a pattern whose wavelength can be determined by linear stability analysis. Dynamic noise, on the other hand, is expected to destabilise this pattern either over time, or very suddenly. Contrary to this intuition, we will now describe three noteworthy effects arising from moderate noise levels.  These effects were validated empirically by simulating the SPDE in \eqref{eq:stochasticMeinhardt} using a finite difference scheme as explained in Appendix \ref{sec:setup} with experimentally calibrated parameters and initial values.

\paragraph{Noise speeds up repolarisation}

In Figure \ref{fig:1}(center) we see the deterministic solution of \eqref{eq:stochasticMeinhardt} (i.e., with $\sigma_A=\sigma_I=0$) starting from a polarised state with high activator concentration in some part of the cell (near $x=0$) called 'rear'. Stimulated by the extracellular signal \eqref{eq:signal} the activator breaks down in order to reappear in an area of high signal strength (near $x=L/2$) called 'front', and the cell repolarises. Figure \ref{fig:1}(right), on the other hand, contains a typical realisation of the SPDE \eqref{eq:stochasticMeinhardt}. The evolution of activator concentration deviates considerably from the deterministic dynamics, but repolarisation is still achieved.  For a more quantitative analysis consider the mean activator concentrations

\begin{align*}
\mu_{F} (t) =  \frac{1}{\abs{S_F}}\int_{S_{F}} A(t,x) dx,\qquad \mu_{R} =  \frac{1}{\abs{ S_R}}\int_{S_{R}} A(0,x) dx,
\end{align*}
near the front $S_{F} = [L\times 0.4, L\times 0.6]$ at $t\geq 0$ and near the rear $S_{R}= [0,L\times 0.1)\cup(L\times 0.9,L]$ at time $t=0$. With this define the 'time to repolarisation'
\[
	\tau_{\gamma} = \inf \{ t\geq0: \, \mu_F (t) \geq \gamma \mu_R \}
\]
\begin{figure}
	\begin{minipage}[t]{0.5\linewidth}
	\centering
	\includegraphics[width=0.8\textwidth]{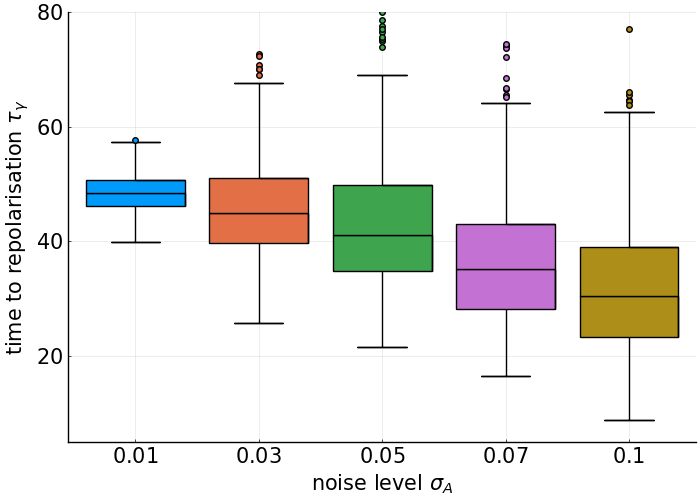}
	\par\vspace{0cm}
	\end{minipage}\hspace{-0.5cm}
	\begin{minipage}[b]{0.5\linewidth}
	\centering	
\resizebox{0.9\textwidth}{!}{
\begin{tabular}{l|cccccccccc}
\toprule
\multirow{2}{2em}{$\sigma_A$} & \multicolumn{5}{c}{$D_A \times 100$}  \\
\cline{2-6}
& 2.42 & 3.42 & $4.42$ & 5.42 & 6.42 \\
\hline
0.00 & 42.7  & 46.0 & 48.9 & 51.0 & 52.0 \\
0.01 & 42.0  & 45.4 & 48.4 & 50.8 & 51.7  \\
0.03 & 37.8 & 41.5  & 45.0 & 47.0 & 49.0  \\
0.05 & 32.5 & 36.3 & 41.1 & 42.5 & 46.2  \\
0.07 & 28.9 & 32.7 & 35.3 & 38.1 & 39.6  \\
0.10 & 24.6 & 28.1 & 30.5 & 32.6 & 33.8 \\
\end{tabular}}
\par\vspace{0.5cm}
\end{minipage}

	\caption{(left) Boxplots for the distribution of the time to repolarisation $\tau_{\gamma}$ for different noise levels $\sigma_A$ with the diffusivity $D_A=4.415\times 10^{-2}$ from Appendix \ref{sec:setup}; (right) median time to repolarisation for different noise levels and diffusivities.}
\label{fig:2}
\end{figure}as the time when the activator concentration in the front part is significantly higher than in the rear at $t=0$ depending on a threshold $\gamma>0$.  Figure \ref{fig:2} displays simulations results for the median times to respolaration obtained after 500 Monte Carlo iterations with $\gamma = 0.5$ for different activator noise levels $\sigma_A$ and across a range of diffusivites close to the value $D_A = 4.415\times 10^{-2}$ from Appendix \ref{sec:setup}. The boxplot in Figure \ref{fig:2}(left) further shows that the median, the upper and the lower quartiles clearly decrease for growing $\sigma_A$ (a linear regression on the median repolarisation times yields $50.775$ as the coefficient of the intercept and $-206.066$ as the slope), while the spread and interquartile range slowly increase. In Figure \ref{fig:2}(right) we further see that $\tau_{\gamma}$ decreases for larger $\sigma_A$ and fixed $D_A$.

This effect of the noise level on the dynamics of $A$ can be understood by a scaling argument.  Assume $\sigma_A>0$ and introduce the process $\check{X}(t)=(\check{A},\check{I})$,
\begin{align}
	\check{A}(t)=\frac{\sqrt{D_A}}{\sigma_A} A\left(\frac{t}{D_A}\right),\qquad \check{I}(t)= I\left(\frac{t}{D_A}\right),\qquad 0\leq t\leq TD_A.\label{eq:scaledA}
\end{align}
Then $\check{A}$ satisfies the SPDE
\begin{align}
\frac{\partial}{\partial t}\check{A}(t,x)=\frac{\partial^{2}}{\partial x^{2}}\check{A}(t,x)+\check{f}_{A}(\check{X}(t,x),x,\sigma_A,D_A)+\check{\xi}_{A}(t,x), \qquad \check{A}(0)=\frac{\sqrt{D_A}}{\sigma_A}A_0,
\label{eq:stochasticMeinhardt_scaled}
\end{align}
with unit diffusivity, unit noise level coefficients, the rescaled nonlinearity
\begin{align*}
	\check{f}_{A}(y,x,\sigma_A,D_A)= \frac{1}{\sigma_A \sqrt{D_A}} f_A\left(\frac{\sigma_A}{\sqrt{D_A}}y_1,y_2,x\right).
\end{align*}
and with a space-time white noise process $\check{\xi}_A$. To see the latter, note that $\xi_A(\cdot/D_A,x)$ and $\check{\xi}_A(\cdot,x)/\sqrt{D_A}$ have the same distribution.  For $y$ restricted to reasonable values close to the initial conditions and using the parameters from Appendix \ref{sec:setup} the derivative of the function $\sigma_A\mapsto \check{f}_{A}(y,x,\sigma_A,D_A)$ is strictly positive and so $\sigma_A\mapsto \check{f}_{A}(y,x,\sigma_A,D_A)$ is increasing for fixed $D_A$.  Since it is the magnitude of the nonlinearity that drives the repolarisation in comparison to the diffusivity induced by the Laplacian $\partial^2/\partial x^2$,  this agrees with the empirical observation from above that a larger noise level $\sigma_A$ speeds up the repolarisation.

The same qualitative results are obtained for parameters in the nonlinearity and initial conditions different from the ones in Appendix \ref{sec:setup}.  Moreover, we have also observed a decrease in $\tau_{\gamma}$ when considering larger inhibitor noise levels $\sigma_{I}$ or a larger signal strength $a$,  while keeping all other parameters fixed.

We conclude that repolarisation is not only stable under noise, but it is even accelerated. An interpretation of this behaviour is that the noise breaks symmetries, making the dynamics more 'turbulent' and therefore the creation of a new front is sped up.

\paragraph{Splitting of the front}

\begin{figure}
	\hspace{0.5cm}\includegraphics[width=0.43\textwidth]{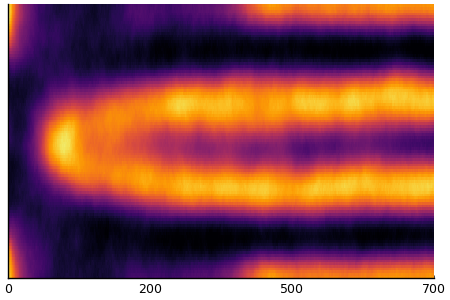}
	\hspace{1cm}
	\includegraphics[width=0.43\textwidth]{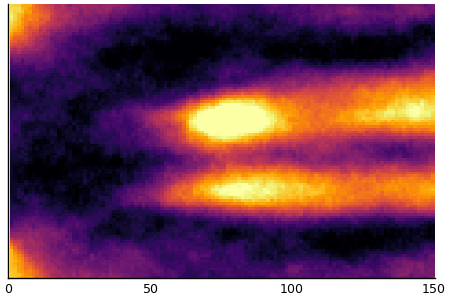}
	\caption{(left) A typical realisation of the stochastic Meinhardt model with a moderate noise level $0.015$ level on a longer time horizon; (right) a realisation of the stochastic Meinhardt model with large noise level $0.05$.}
\label{fig:4}
\end{figure}

For the deterministic Meinhardt model it has been shown by  \citep{lockley2017image} that the repolarised front may not be stable. Indeed, if the parameters, obtained from fitting data on the short timescales at which repolarisation typically occurs (120 sec), are used for long term simulations, then the front splits into several parts. This can be verified for the parameters in Appendix \ref{sec:setup}: upon repolarisation, the front splits first into two parts (around time $t = 200$) and then into three parts (around time $t = 700$).

This behaviour can still be observed in the stochastic Meinhardt model with small noise levels (cf. Figure \ref{fig:4}(left)), but both the splitting into two and also into three fronts happens much faster than in the deterministic system.  For example,  for $\sigma_{A} = 0.015$,  the front splitting into three fronts occurs already at $t = 400$ - $500$ as compared to $t=700$ for $\sigma_{A}=0$.  More strikingly,   larger noise levels may even lead to repolarisation with a sudden split without ever achieving a single stable front (cf. Figure \ref{fig:4}(right)), which has not been observed in the deterministic model before, even with different parameter choices.  While not all simulated paths show the same splitting pattern,  the intensity of sudden splits increases with growing $\sigma_A$,  with $\sigma_A=0.03$ appearing to be a threshold above which almost all paths split immediately.

Let us remark that front splitting is a common feature of amoeboid cell migration, allowing cells to explore their environment and to respond to changes in gradients more quickly.  The front experiencing a stronger increase in signal will be enhanced while the other will be retracted.  In strong signal gradients cells can move with a single front for long times (>10 minutes) \citep{Lockley_2015}. The Meinhardt model for long term simulations requires a smaller diffusivity $D_{I}$. When $D_{I}$ was reduced by $25 \%$, both deterministic and stochastic solutions produced a single stable front.

\paragraph{Diffusion interacts with noise}
\begin{figure}
	\includegraphics[width=\textwidth]{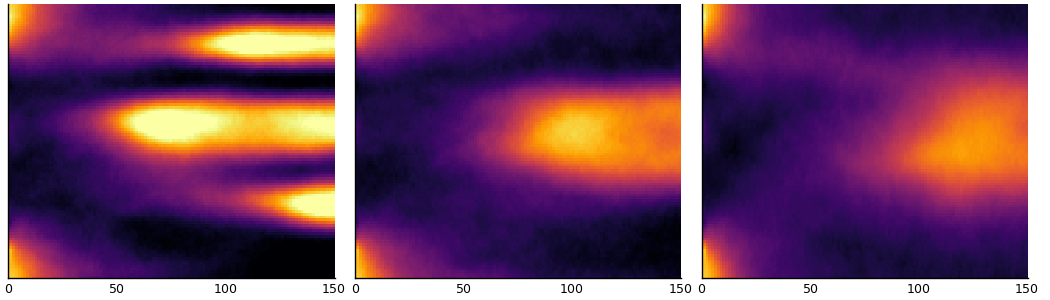}
	\caption{Three different realisations of the stochastic Meinhardt model with moderate noise level $0.02$ and with diffusivity $D_A - 0.01$ (left),  $D_A$ (center) and $D_A + 0.03$ (right), where $D_A= 4.415 \times 10^{-2}$ is as in Appendix \ref{sec:setup}.}
\label{fig:3}
\end{figure}

The rescaling leading to \eqref{eq:stochasticMeinhardt_scaled} reveals that $\sigma_A/\sqrt{D_A}$ is an important factor driving the dynamics of $A$. It suggests that larger diffusivity levels $D_A$ may counteract an increase in $\sigma_A$,  at least if we ignore the rescaling of the time horizon in \eqref{eq:scaledA}.  This can be seen from Figure \ref{fig:3}. First, the left picture shows that already a small noise level combined with a smaller diffusivity may exhibit the sudden front splits discussed in the last paragraph in the context of large noise levels (cf. Figure \ref{fig:4}(right)).  Second, comparing the plots from left to right, we notice that a larger diffusivity leads to a slower repolarisation, which is is line with the discussion from the previous paragraph.  Indeed, if $\sigma_A/\sqrt{D_A} \equiv \bar{\sigma}$ is kept constant for some $\bar{\sigma}>0$, then $D_A\mapsto \check{f}_{A}(y,x,\bar{\sigma}\sqrt{D_A},D_A)$ is clearly decreasing.  The interaction effect between $D_A$ and $\sigma_A$ is also visible in Figure \ref{fig:2}(right), since the slow down in time to  repolarisation for larger $D_A$ depends on the magnitude of the noise.

At last, let us mention that the ratio $\sigma_A/\sqrt{D_A}$ also appears in the limiting observed Fisher information (cf.  Proposition \ref{prop:FisherInfo} below) of our estimation method discussed in the next section. This means that a larger noise to diffusivity ratio not only corresponds to faster repolarisation, but also to an increase in information leading to smaller confidence intervals.

\section{Parameter estimation}\label{sec:estimation}

We derive an estimator $\hat D_{A,\delta}$ of the diffusivity $D_A$ from first principles and state its main properties. Let us assume that we can measure the activator $A$ at $M$ points $x_{k}\in\Lambda$ for $k=1,\dots,M$ over a period of time $[0,T]$.  The inhibitor $I$ can typically not be measured. Measurements of $A$ correspond to fluorescence distributions (for example of actin in \citep{Lockley_2015}) at the cell cortex and are obtained through microscopy. This means that every measurement necessarily has a minimal spatial resolution $\delta>0$ determined by the experimental setup. It can be described by a \emph{local measurement} \citep{Altmeyer2019b}, that is, a linear functional
\begin{equation}
	A_{\delta}(t,x_{k}):=A(t)*K_{\delta}(x_k)=\sc{A(t)}{K_{\delta,x_{k}}},\quad 0\leq t\leq T,\label{eq:localMeasurement}
\end{equation}
where $\sc{\cdot}{\cdot}$ denotes the $L^2(\Lambda)$ inner product and $*$ means convolution with respect to $K_{\delta}(x):=\delta^{-1/2}K(\delta^{-1}x)$
for a compactly supported function $K\in H^2(\R)$, the classical $L^2$-Sobolev space of order 2. Moreover, $K_{\delta,x_k}:=K_{\delta}(\cdot - x_k)$ corresponds to the point spread function in imaging systems. Particular examples for $K$ are bump functions (see Section \ref{sec:numerics} below). The scaling by $\delta^{-1/2}$ is irrelevant for the estimator, but normalizes $K_{\delta,x_k}$ in $L^2$-norm, that is, $\norm{K_{\delta,x_k}}=\norm{K}_{L^2(\R)}$, and eases the notation later. In contrast to the common setting in statistics for SPDEs we thus dispose here only of partial observations given by local measurements of one component $A$ in the presence of a hidden component $I$.

From \eqref{eq:stochasticMeinhardt} we find that $A_{\delta}(t,x_k)$ satisfies
\begin{equation}
\frac{\partial}{\partial t}A_{\delta}(t,x_{k})= D_{A}A_{\delta}^{\Delta}(t,x_{k})+\sc{f_A(X(t,\cdot),\cdot)}{K_{\delta,x_{k}}} +\sigma_{A}\norm{K}_{L^2(\R)}\, \xi_{A,k}(t),\label{eq:A_obs_equation}
\end{equation}
with scalar white noise (in time) $\xi_{A,k}(t)=\sc{\xi_A(t)}{K_{\delta,x_{k}}}/\norm{K}_{L^2(\R)}$, and where
\begin{equation}
A_{\delta}^{\Delta}(t,x_{k}):=\frac{\partial^{2}}{\partial x^{2}}A_{\delta}(t,x_k)=\Big\langle A(t),\frac{\partial^{2}}{\partial x^{2}}K_{\delta,x_{k}}\Big\rangle.\label{eq:LaplaceMeasurement}
\end{equation}
Neglecting the contribution of the nonlinear term in \eqref{eq:A_obs_equation} leads to a parametric estimation problem for $D_A$ with respect to the scalar processes $(A_{\delta}(t,x_k))_{0\leq t\leq T}$ for $k=1,\dots,M$. The maximum-likelihood estimator can be obtained, in principle, by Girsanov's theorem \citep{Liptser2001}, but this leads to a non-explicit filtering problem, as explained in \citep{Altmeyer2019b} for the case $M=1$. Instead, consider the modified likelihood with stochastic differentials $dA_{\delta}(t,x_{k})$ (in time)
\[
	\c L_{\delta}(D_A) = \exp\left(\frac{D_A}{\sigma_{A}^2 \norm{K}_{L^2(\R)}} \sum_{k=1}^M \left( \int_0^T  A^{\Delta}_{\delta}(t,x_k) dA_{\delta}(t,x_k)-\frac{D_A}{2}\int_0^T \left( A^{\Delta}_{\delta}(t,x_k)\right)^2 dt\right) \right).
\]
Maximising with respect to $D_A$ and assuming that we have also measurements $(A_{\delta}^{\Delta}(t,x_{k}))_{0\leq t\leq T}$ at our disposal, leads to the  \emph{augmented MLE}
\begin{equation}
\hat{D}_{A,\delta}=\frac{\sum_{k=1}^{M}\int_{0}^{T}A_{\delta}^{\Delta}(t,x_{k})dA_{\delta}(t,x_{k})}{\sum_{k=1}^{M}\int_{0}^{T}(A_{\delta}^{\Delta}(t,x_{k}))^{2}dt}.\label{eq:D_A}
\end{equation}
This extends the construction of \citep{Altmeyer2019b} to more than one pair of local measurements. Equivalently, $\hat{D}_{A,\delta}$ can
be obtained formally (that is, neglecting the term independent of $D_A$ in the quadratic expansion and interpreting $\frac{\partial}{\partial t} A_\delta dt=dA_\delta$) as minimiser of the least squares contrast
\begin{align*}
 & D_{A}\mapsto\sum_{k=1}^{M}\int_{0}^{T}\left(\tfrac{\partial}{\partial t} A_{\delta}(t,x_{k})-D_{A}A_{\delta}^{\Delta}(t,x_{k})\right)^{2}dt.
\end{align*}
With Brownian motions $W_k(t)=\int_0^t\xi_{A,k}(s)ds$ we obtain from \eqref{eq:A_obs_equation} the basic error decomposition
\begin{align}
\hat{D}_{A,\delta}
	& =D_{A}+\c I_{\delta}^{-1}\c R_{\delta}+\sigma_{A}\norm K_{L^{2}(\R)}\c I_{\delta}^{-1}\c M_{\delta},\label{eq:errorDecomp}\\
	& \text{with} \quad \c I_{\delta}  =\sum_{k=1}^{M}\int_{0}^{T}\left(A_{\delta}^{\Delta}(t,x_{k})\right)^{2}dt,\qquad\qquad  \text{(observed Fisher information)}\nonumber \\
	& \quad \,\,\,\,\quad \c M_{\delta} =\sum_{k=1}^{M}\int_{0}^{T}A_{\delta}^{\Delta}(t,x_{k})dW_{k}(t), \quad\quad\quad\qquad\qquad \text{(martingale part)}\nonumber\\
	&\quad \quad\quad \c R_{\delta}  =\sum_{k=1}^{M}\int_{0}^{T}A_{\delta}^{\Delta}(t,x_{k})\sc{f_A(X(t,\cdot),\cdot)}{K_{\delta,x_{k}}}dt. \quad \text{(nonlinear bias)}\nonumber
\end{align}
For $M=1$ and linear SPDEs with Dirichlet boundary conditions  \citep{Altmeyer2019b} show that $\c I_{\delta}\rightarrow \infty$ in probability for resolution $\delta\rightarrow 0$. We will see that this remains true in the present case with periodic boundary conditions and  fixed $M$. For independent Brownian motions $W_k$, for example when the $K_{\delta,x_k}$ have disjoint supports, the observed Fisher information $\c I_{\delta}$ corresponds to the quadratic variation of the martingale part $\c M_{\delta}$. Consistency of $\hat{D}_{A,\delta}$ is therefore expected to hold as soon as the nonlinear bias is not too large. This is shown by \citep{Altmeyer2020a} for sufficiently regular nonlinearities depending only on the observed process.  

We prove in Appendix
\ref{subsec:Results-on-estimation} for fixed $T$ and $M$ that $\hat D_{A,\delta}$ is not only a consistent estimator of $D_{A}$  for resolution levels $\delta\rightarrow0$, but also that its error satisfies a central limit theorem with rate $\delta$ (which is optimal already for the linear case in \citep{Altmeyer2019b}) and with explicit asymptotic variance.
\begin{theorem}
\label{thm:estimate_DA} Consider the setting of Theorem \ref{thm:existenceUniqueness} and let $K\in H^{2}(\R)$, $K \neq 0$, have compact support. If $\sigma_A>0$, then $\hat{D}_{A,\delta}$ is a consistent
and asymptotically normal estimator of $D_{A}$, more precisely as $\delta\rightarrow0$
\[
\delta^{-1}\left(\hat{D}_{A,\delta}-D_{A}\right)\r dN\left(0,D_{A}\frac{\Sigma}{MT}\right),\quad\Sigma=\frac{2\norm K_{L^{2}(\R)}^{2}}{\norm{\frac{\partial}{\partial x}K}_{L^{2}(\R)}^{2}}.
\]
\end{theorem}
The asymptotic variance in Theorem \ref{thm:estimate_DA} decreases for more observations $M$ and for a growing time horizon $T$, but is independent of the noise level $\sigma_A$, the initial value, the nonlinearity $f_A$ and the inhibitor $I$. This robustness is particularly important in modelling realistic nonlinear dynamics such as \eqref{eq:stochasticMeinhardt}, which are subject to model uncertainties in parameters and even in the form of the equation. In fact, the proof reduces the estimation problem to the linear case (where $f_A=0$) and with zero initial value. This is possible, because the nonlinear bias is asymptotically of much smaller order than the martingale part in \eqref{eq:errorDecomp}. The dependence on $\delta$ is comparable to the presence of a zero order term in a linear SPDE as opposed to a first order linearity, that is, a transport term, which can induce an asymptotic bias \citep{Altmeyer2019b}.
\begin{remark}\label{rem:generalProof}
The proof of Theorem \ref{thm:estimate_DA} is inspired by Theorem 5.3 of  \citep{Altmeyer2019b} and Theorem 3 of  \citep{Altmeyer2020a} for parametric diffusivity and spatially homogeneous noise, but  is significantly shorter and more transparent.  While we focus here on the stochastic Meinhardt model \eqref{eq:stochasticMeinhardt},  the proof of the theorem applies without any changes to arbitrary (even unbounded) domains in $\R^d$, general boundary conditions and any nonlinearity $f_A$ satisfying $t\mapsto f_{A}(X(t,\cdot ),\cdot)\in C([0,T];C(\Lambda))$, $\tilde{A}\in C([0,T];C^{2}(\Lambda))$ $\P$-almost surely for $\tilde{X}=(\tilde{A},\tilde{I})$, where $\tilde{X}$ is the perturbation process in Theorem \ref{thm:existenceUniqueness}.

\end{remark}

\begin{remark}\label{rem:noNoise}
It is interesting that the robustness of the estimator $\hat{D}_{A,\delta}$ to nonlinear perturbations $f_A$ is an impact of the driving noise process. If there is no noise, that is, $\sigma_A=\sigma_I=0$, then $A(t)\in C^{2}(\Lambda)$, $f_A(X(t,\cdot),\cdot)\in C(\Lambda)$ by classical theory for parabolic PDEs \citep{Evans2010} or argued as in the proof of Theorem \ref{thm:existenceUniqueness}, which does not assume nonvanishing noise. This implies by convolution approximation uniformly in $0\leq t\leq T$ as $\delta\rightarrow 0$
\begin{align*}
	& \sc{f_A(X(t,\cdot),\cdot)}{\delta^{-1/2}K_{\delta,x_{k}}}\rightarrow f_{A}(X(t,x_k),x_k)\int_{\R}K(x)dx,\\
	& \delta^{-1/2}A_{\delta}^{\Delta}(t,x_{k}) =\Big\langle \frac{\partial^{2}}{\partial x^{2}}A(t,\cdot),\delta^{-1/2}K_{\delta,x_{k}}\Big\rangle \rightarrow \frac{\partial^{2}}{\partial x^{2}}A(t,x_k)	\int_{\R}K(x)dx.
\end{align*}
From this and the basic error decomposition \eqref{eq:errorDecomp} it follows then, assuming $\int_{\R}K(x)dx\neq 0$ and $\frac{\partial^{2}}{\partial x^{2}}A(t,x_k)\neq 0$ for at least on $x_k$, that $\hat{D}_{A,\delta}-D_{A}$ converges to a non-zero constant. On the other hand, in the linear PDE case with $f_A= 0$ and with $\sigma_A=\sigma_I=0$, we have exactly $\hat{D}_{A,\delta}=D_A$ and there is no estimation error.
\end{remark}
As a consequence of Theorem \ref{thm:estimate_DA} and Slutsky's Lemma we can easily construct asymptotic confidence intervals for $D_A$.
\begin{corollary}\label{cor:confInterval1}
Consider the setting of Theorem \ref{thm:estimate_DA}. For $0<\alpha<1$ an asymptotic
confidence interval for $D_{A}$ with asymptotic
coverage $1-\alpha$ as $\delta\rightarrow0$ is given by
\begin{align*}
I_{1-\alpha} & =\left[\hat{D}_{A,\delta}-\frac{\delta}{(MT)^{1/2}}\left(\hat{D}_{A,\delta}\Sigma\right)^{1/2}q_{1-\alpha/2},\hat{D}_{A,\delta}+\frac{\delta}{(MT)^{1/2}}\left(\hat{D}_{A,\delta}\Sigma\right)^{1/2}q_{1-\alpha/2}\right],
\end{align*}
with the standard normal $(1-\alpha/2)$-quantile $q_{1-\alpha/2}$.
\end{corollary}

While the confidence interval $I_{1-\alpha}$ requires the knowledge of $K$ as part of $\Sigma$ from Theorem \ref{thm:estimate_DA},  we can obtain a fully data driven confidence interval.  Indeed,  noting that the quadratic variation $QV=T\sigma_A^2 \norm{K}^2_{L^2(\R)}$ of $A_\delta(\cdot,x_k)$ in \eqref{eq:A_obs_equation} at time $T$ is identified from observing the trajectories of $(A_{\delta}(t,x_k))_{0\leq t\leq T}$ in continuous time and observing that Proposition \ref{prop:FisherInfo}(i,ii,iii) below provides the convergence $\delta^2\mathcal{I}_{\delta}\rightarrow \,QV\cdot M\cdot(D_A \Sigma)^{-1}$ in probability, we obtain the following completely data-driven confidence intervals.

\begin{corollary}\label{cor:confInterval2}
As in Corollary \ref{cor:confInterval1}, for $0<\alpha<1$ an asymptotic
confidence interval for $D_{A}$ with asymptotic
coverage $1-\alpha$ as $\delta\rightarrow0$ is given by
\begin{align*}
\tilde{I}_{1-\alpha} & =\left[\hat{D}_{A,\delta}-\left(\frac{QV}{T\c I_{\delta}}\right)^{1/2}q_{1-\alpha/2},\hat{D}_{A,\delta}+\left(\frac{QV}{T\c I_{\delta}}\right)^{1/2}q_{1-\alpha/2}\right].
\end{align*}
\end{corollary}

\section{Application to synthetic data}\label{sec:numerics}

\begin{figure}

\includegraphics[width=1.0\linewidth,valign=m]{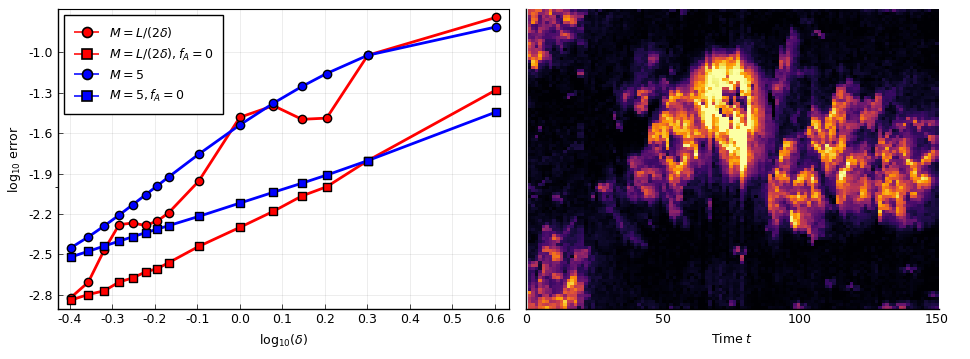}

\caption{ {\label{fig:5} (left) $\log_{10}$-$\log_{10}$ plot of root mean squared estimation errors for different $M$ with $f_A$ and $A_0$ calibrated to experimental data and in the linear case with $f_A=0$ and zero initial condition $A_0$; (right) heatmap for 100 measurements of the activator along the cell contour for a single cell. }}

 \end{figure}

Let us apply the results of the previous section.  Synthetic data of local measurements are obtained by simulating the SPDE \eqref{eq:stochasticMeinhardt}  as in Section \ref{sec:effectofnoise} with experimentally calibrated $D_A=4.415 \times 10^{-2}$ and with $\sigma_A=0.02$, $\sigma_I=0$.  As a typical example for the kernel $K$ we use the bump function
\begin{equation}
	K(x) = \exp\left(-\frac{10}{1-x^2}\right)\mathbf{1}_{[-1,1]}(x), \quad x\in\R.
\end{equation}
For different resolutions $\delta\in [L \times 0.017,L \times 0.1]$ and different $M$, local measurements $A_{\delta}(t_j,x_k)$, $A^{\Delta}_{\delta}(t_j,x_k)$ are obtained according to \eqref{eq:localMeasurement} and \eqref{eq:LaplaceMeasurement} on regular grids $x_k=Lk/M$, $k=0,\dots,M-1$, $t_j = Tj/N$, $j=0,\dots,N$. For these local measurements, the augmented MLE $\hat{D}_{A,\delta}$ is computed.

Figure \ref{fig:5}(left) shows a $\log_{10}$-$\log_{10}$ plot of root mean squared estimation errors for $\hat D_{A,\delta}$ obtained after 500 Monte Carlo iterations  for $T=30$, $L=20$, $m=2000$ points in space and $n=m^2/4$ points in time, with $N=n/100$. We show results for two different choices of $M$, once fixed at $M=5$ for all $\delta$ and once $M\equiv M(\delta)=L/(2\delta)$. In both cases, the supports of the kernels $K_{\delta,x_k}$ are non-overlapping. For comparison, we also added root mean squared estimation errors for the linear SPDE with $f_A=0$ and with zero initial value $A_0=0$. 

The estimation errors are significantly smaller in the linear case,  for both choices of $M$, therefore demonstrating clearly a strong non-asymptotic effect of the nonlinearity and the initial value. This difference disappears as $\delta\rightarrow 0$. In the linear case, the errors are very well aligned with the asymptotic standard error (when multiplied with $\delta$) as predicted by Theorem \ref{thm:estimate_DA}. This allows us to read off the rates of convergence from the $\log_{10}$-$\log_{10}$ plots: for fixed $M$ the rate of convergence is $\delta$, and for $M(\delta)=L/(2\delta)$ it is approximately $\delta/\sqrt{M}\approx \delta^{3/2}$. These rates are not attained yet for the considered range of $\delta$ with non-vanishing $f_A$ and with a non-zero initial value. Note that for $\delta=L \times 0.017$ we only obtain $M(\delta)=30$ non-overlapping local measurements, whereas in the real-data example in the next section we have a much smaller $\delta$ and 100 measurements,  so a better performance can be expected.  Finer Monte-Carlo simulations would require significantly more computational efforts.

We have also verified the confidence intervals $I_{1-\alpha}$ and $\tilde{I}_{1-\alpha}$ of Corollaries \ref{cor:confInterval1} and \ref{cor:confInterval2} empirically for different $\alpha$.  $\tilde{I}_{1-\alpha}$ is always smaller than $I_{1-\alpha}$ due to the strong non-asymptotic effect of the nonlinearity leading to overestimation of the asymptotic Fisher information (cf. Proposition \ref{prop:FisherInfo}(i)).  Again, this difference disappears as $\delta\rightarrow 0$.  Because of this,  coverage near the nominal level is obtained in the linear case for all $\delta$,  while in the nonlinear case good coverage requires relatively small $\delta$.  For example,  with respect to $M(\delta)$, $\delta=L \times 0.017$ and $\alpha=0.1$,  the interval $I_{90}$ covers $D_A$  for 85\% of the samples, and for $\alpha=0.05$ with $I_{95}$ this increases to 92\%. In one sample, with $D_A$ was estimated at $4.372 \times 10^{-2}$ with confidence bounds $\pm 0.162 \times 10^{-2}$ according to $I_{90}$, and with $\pm 0.193 \times 10^{-2}$ according to $I_{95}$. Moreover,  a larger time horizon $T$ decreases the overall estimation errors, as predicted by Theorem \ref{thm:estimate_DA},  but coverage in the nonlinear case is again adversely affected,  which indicates that the error induced by the nonlinearity in estimating the asymptotic Fisher information (cf.  equation \eqref{eq:A_tilde}) may not decrease for larger $T$. Again,  coverage improves for smaller $\delta$.


Further unreported simulations show that pointwise estimation results, that is, with $M=1$, are not homogeneous in space, and are affected adversely at locations $x_k$ where repolarisation leads to fast changes in the activator $A$ (cf. Figure \ref{fig:1}). This effect becomes smaller as $\delta\rightarrow 0$ because the nonlinearity plays no role in the asymptotic error according to Theorem \ref{thm:estimate_DA}. While the augmented MLE can account for these local fluctuations, we have noticed that a discretised version of the spectral estimator \citep{Huebner1995}, obtained from a discrete Fourier transform of the local measurements, is affected considerably by the nonlinearity and does not perform well, unless the number of measurements $M$ is increased significantly. This leads, however, to overlapping supports of the $K_{\delta,x_k}$, which may not be realistic in experimental data, for example in case of the imaging data in the next section. For more details and different aspects of parameter estimation in stochastic reaction-diffusion equations using the spectral estimator see  \citep{Pasemann2020} in a related application to cell motility.

\section{Application to experimental data}\label{sec:realData}

We apply the augmented MLE to experimental single cell data from 18 different single cells as described in Appendix \ref{sec:setup}.  Figure \ref{fig:5}(right) shows the heatmap for one such cell. Compared to the average over these 18 cells displayed in Figure \ref{fig:1}(left), the activator behaves much more random, with cell fronts  forming and disappearing quickly.


We assume that each data point corresponds to a local measurement $A_{\delta}(t_j,x_k)$ for $x_k = Lk/M$, $k=0,\dots,M-1$ for $M=100$ and with $t_j=Tj/N$, $j=0,\dots,N$. Here, $T=N$ ranges from $30$ to $256$ for different cells and we choose $L=20$, as above. Using \eqref{eq:LaplaceMeasurement}, $A^{\Delta}_{\delta}(t_j,x_k)$ is computed by a finite difference approximation of the second derivative. This yields an approximated version of the augmented MLE $\hat{D}_{A,\delta}(\pm \hat{\kappa}_{1-\alpha})$, where $\tilde{I}_{1-\alpha}=[\hat{D}_A - \hat{\kappa}_{1-\alpha}, \hat{D}_A + \hat{\kappa}_{1-\alpha}]$ is the confidence interval from Corollary \ref{cor:confInterval2}, which is obtained from the observed Fisher information given the approximated $A^{\Delta}_{\delta}(t_j,x_k)$ and from using the averaged realised variations
\begin{equation*}
	RV = M^{-1}\sum_{k=0}^{M-1} \sum_{j=1}^N (A_{\delta}(t_j,x_k)-A_{\delta}(t_{j-1},x_k))^2
\end{equation*}
to approximate the quadratic variation $QV = T\sigma_A^2 \norm{K}^2_{L^2(\R)}$.

For the single cell data displayed in Figure \ref{fig:5}(right) we estimate $D_A$  at $T=50$ as $1.450 \times 10^{-2}(\pm 0.048\times 10^{-2})$, at $T=150$ as $1.600 \times 10^{-2}(\pm 0.031\times 10^{-2})$ and at $T=250$ as $1.620 \times 10^{-2}(\pm 0.024\times 10^{-2})$.  This shows that the estimates stabilise for growing $T$ with lower and upper confidence bounds decreasing approximately according to $1/\sqrt{T}$ as indicated by Theorem \ref{thm:estimate_DA}.  Across the 18 cells we obtain similar results.  The mean estimated diffusivity at $T=30$ (which is the shortest time length among the 18 cells) is $1.76\times 10^{-2}$. We therefore obtain estimated diffusivities of a comparable order to previously reported results in the literature, for example by \citep{fukushima2019excitable} in a similar (deterministic) setup, but also  by  \citep{Lockley_2015}, who fitted the same data used in the present paper to the deterministic Meinhardt model by a least squares (profile likelihood) approach.  Interestingly,  \citep{Lockley_2015} obtain much smaller diffusivities for the single cell data than for the averaged data across the 18 cells displayed in Figure  \ref{fig:1}(left).  The augmented MLE, however, yields for the averaged data $1.9\times 10^{-2}$ (until $T=30$), which is close to the averaged estimated diffusivities from above.  We conclude that the augmented MLE is robust to the averaging of data for different cells with similar diffusivities.  On the other hand,  averaging for different cells yields significantly smaller confidence bounds. This can be expected in our setup, since it is reasonable to assume that the driving noise processes $\xi_{A}$ in \eqref{eq:stochasticMeinhardt} are independent for different cells.  In this case,  as the nonlinearity $f_A$ plays only a minor role in estimating $D_A$ as discussed after Theorem \ref{thm:estimate_DA},  the averaged observations satisfy an SPDE with the same diffusivity, but with the modified noise level $\sigma_A/\sqrt{18}$.

The time to repolarisation in the 18 cells (i.e., $\tau_{\gamma}$ with $\gamma = 0.5$) ranges from $15$ to $49$ with median $22$, which are reasonable values for \textit{Dictyostelium} cells  \cite{dalous2008reversal}. We have also simulated the stochastic Meinhardt model using the averaged estimated diffusivities, keeping the same parameters for $f_A$ and the same initial conditions as in Section \ref{sec:numerics}, and computed the time to repolarisation as in Section \ref{sec:effectofnoise}. The behaviour of $\tau_\gamma$ is similar, with its median decreasing for growing $\sigma_{A}$. For example, for $\gamma = 0.5$, we get $\tau_{\gamma} = 40.38$ in the deterministic case, the median value $\tau_{\gamma} = 36.89$ when $\sigma_{A} = 0.02$ and the median value $\tau_{\gamma} = 23.29$ for $\sigma_{A} = 0.10$.

Since the estimated diffusivities are smaller than the diffusivity used for our previous simulations, several new fronts were typically built upon repolarisation. As discussed in Section \ref{sec:effectofnoise}, the reduction of $D_{I}$ helps to create one stable front again.

\section{Discussion}\label{sec:discussion}

We have extended parameter estimation methods developed for linear SPDEs to systems of stochastic reaction-diffusion equations with periodic boundary conditions. The inclusion of noise into biological models is becoming increasingly relevant, owing to the availability of high resolution measurement devices and improved computational methods.


As a concrete application we have estimated the diffusivity in a stochastic Meinhardt model for cell repolarisation. The estimator performed well on synthetic data and  provided reasonable estimates across measurements for 18 single cells. For the considered SPDE model, we have demonstrated through simulations that moderate levels of dynamic noise do not destroy the pattern formation mechanism, but amplify it, leading to faster repolarisation and front splitting. This is achieved despite the simple activator-inhibitor structure of \eqref{eq:stochasticMeinhardt} and using only space-time white noise. We believe that this is the starting point for studying more detailed models for cell repolarisation based on SPDEs  with spatially nonhomogeneous and possibly multiplicative noise. In this way we hope to obtain models that recover the variations within cells and between different cell populations better.

The estimation methods developed here are not limited to the specific SPDE model under consideration, but can also be applied to other models for cell motility such as  \citep{fukushima2019excitable, Pasemann2020}, and even to general systems of stochastic reaction-diffusion equations under regularity conditions for the nonlinearity. This flexibility will be essential in calibrating SPDE models to experimental data.


\appendix
\section{Proofs}\label{subsec:Existence-of-a}

In the following, we consider for fixed $T<\infty$ a filtered probability space $(\Omega,\c F, (\c F_t)_{0\leq t\leq T},\P)$, where the filtration $(\c F_t)_{0\leq t\leq T}$ is generated by the independent Brownian motions $W_k$ in \eqref{eq:errorDecomp}.  Unless stated otherwise, all limits are taken as $\delta \rightarrow 0$. $C$ always denotes a generic positive constant which may change from line to line. $A\lesssim B$ means $A\leq CB$ and $A_n = \c O_{\P}(B_n)$ means that $A_n/B_n$ is tight, that is, $\sup_n\P(|A_n| > C|B_n|) \rightarrow 0$ as $C\to\infty$. Recall that $z\in L^{p}(\Lambda)=L^{p}(\Lambda;\R)$  for $p\geq 1$ and  $\Lambda=\R/(L\Z)$ means that $z$ is $L$-periodic and $z\in L^p([0,L])$. We also write $\Delta=\partial^{2}/\partial x^{2}$ to
denote the Laplacian on $L^{2}(\Lambda)$ with periodic boundary conditions.

\subsection{Existence of a unique solution}\label{sec:mild solution}

\paragraph{Reformulation and mild solution}

Let us first reformulate the Meinhardt model as an SPDE in the space
$L^2(\Lambda; \R^2)$.  Let $S_A$ and $S_I$ denote the analytic and self-adjoint semigroups generated by $D_{A}\Delta$ and $D_{I}\Delta$ on $L^2(\Lambda)$, cf. \citep[Section 2.3]{berglund2019introduction}. For smooth $z=(z_{1},z_{2})\in L^2(\Lambda; \R^2)$ consider also the differential operator $\c Az =\left(D_{A}\Delta z_{1},D_{I}\Delta z_{2}\right)$ with periodic boundary conditions, generating the semigroup $S(t)z=(S_A(t)z_1,S_I(t)z_2)$ on $L^2(\Lambda; \R^2)$.  Let $B:L^{2}(\Lambda;\R^{2})\rightarrow L^{2}(\Lambda;\R^{2})$,
$Bz=(\sigma_{A}z_{1},\sigma_{I}z_{2})$ and define $F:L^{2}(\Lambda;\R^{2})\rightarrow L^{2}(\Lambda;\R^{2})$ by
\[
F(z)(x) = (F_A(z)(x), F_I(z)(x)) = \left(f_A(z(x),x),f_I(z(x),x)\right).
\]
Consider two independent cylindrical Wiener processes $W_A$, $W_I$ on $L^2(\Lambda)$ such that $W(t)=(W_A(t),W_I(t))$ is a cylindrical Wiener process on $L^2(\Lambda;\R^2)$, and formally $dW(t)=(\xi_{A}(t),\xi_{I}(t))dt$. Solving \eqref{eq:stochasticMeinhardt} then corresponds to finding a solution   $X=(A,I)$ in $L^2(\Lambda;\R^2)$ to the SPDE
\begin{equation}
\begin{cases}
dX(t)=\left(\mathcal{A}X(t)+F(X(t))\right)dt+B\,dW(t),\quad0<t\leq T,\\
X(0)=(A_{0},I_{0}).
\end{cases}\label{eq:SPDE}
\end{equation}
We use the mild solution concept of \citep{DaPrato2014}. We will show that there exists a process $X$ taking values in $L^{2}(\Lambda;\R^{2})$ satisfying
\begin{equation}
X(t)=S(t)X(0)+\int_{0}^{t}S(t-s)F(X(s))ds+\int_{0}^{t}S(t-s)BdW(s).\label{eq:X}
\end{equation}
%

\paragraph{Linear and nonlinear parts}

The idea is to obtain the existence of $X$ from $X:=\bar{X}+\tilde{X}$ with the stochastic convolution $\bar{X}(t):=\int_{0}^{t}S(t-s)BdW(s)$ and where $\tilde{X}$ satisfies
\begin{equation}
\tilde{X}(t)=S(t)X(0)+\int_{0}^{t}S(t-s)F(\bar{X}(s)+\tilde{X}(s))ds,\quad0\leq t\le T.\label{eq:X_tilde}
\end{equation}
The process $\bar{X}$ is the unique mild solution to the linear SPDE \eqref{eq:SPDE} (with $F\equiv 0$ and $X(0)=0$) and takes values in $L^{2}(\Lambda;\R^{2})$ (apply, for example, \citep[Theorem 5.4]{DaPrato2014} separately to the component processes $\bar{A}$, $\bar{I}$). Finding a process $\tilde{X}$ solving \eqref{eq:X_tilde}, on the other hand, means equivalently finding a solution to the nonlinear PDE with random coefficients
\begin{equation}
\frac{\partial}{\partial t}\tilde{X}(t)=\c A\tilde{X}(t)+F(\bar{X}(t)+\tilde{X}(t)),\quad0<t\leq T,\quad\tilde{X}(0)=X(0).\label{eq:X_tilde_pde}
\end{equation}
Since this equation does not depend explicitly on the noise process $W$ anymore,
it can be solved for a fixed realisation of $\bar{X}$. The proof follows from a classical fixed point argument.
\begin{proof}[Proof of Theorem \ref{thm:existenceUniqueness}]
By a fixed point argument \citep[Theorem 6.4]{Hairer2009}, noting that $X(0)\in C(\Lambda;\R^2)$ and $F$ is globally Lipschitz continuous from $C(\Lambda;\R^2)$ to $C(\Lambda;\R^2)$, we conclude that \eqref{eq:X} and \eqref{eq:X_tilde} have unique solutions $X$ and $\tilde{X}$ in $C([0,T];C(\Lambda;\R^2))$. In order to obtain the higher regularities, we first show that the linear process $\bar{X}$ takes values in $C^{s'}(\Lambda;\R^2)$, $0<s'<1/2$. For $r\in\R,\,p\geq1$ consider the Bessel potential spaces on the 1D-torus $\Lambda=\R/(L\Z)$
\[
W^{r,p}(\Lambda):=\left\{ u\in L^{p}(\Lambda):\norm u_{r,p}<\infty\right\},
\]
with norm $\norm u_{r,p}=\norm{(I-\Delta)^{r/2}u}_{L^{p}(\Lambda)}$ \citep{Triebel2010a,Cirant2019}. Note that $I-\Delta$ is a strictly positive operator and thus $(I-\Delta)^{-1}$ is a bounded operator on $L^{p}(\Lambda)$
for periodic boundary conditions, while $(-\Delta)^{-1}$ is not. The Bessel potential spaces differ from the classical Sobolev spaces, but allow for a Sobolev embedding theorem (see for example \citep[Section 2.3]{Cirant2019} or \citep{Triebel2010a}). We can now apply \citep[Proposition 30]{Altmeyer2020a} to conclude that $\bar{A}$, $\bar{I}\in C([0,T];W^{s',p}(\Lambda))$ for all $0<s'<1/2$ and $p\geq 2$ for the component processes of $\bar{X}$. Note that the proposition is stated for the Bessel potential spaces with respect to the Dirichlet Laplacian, but the statement and its proof remain true using the corresponding spaces defined above. Since $p$ is arbitrary, the Sobolev embedding applied componentwise gives $\bar{X}\in C([0,T];C^{s'}(\Lambda;\R^2))$ for all $0<s'<1/2$, as claimed.  Next,  introduce the spaces
\[
	W^{r,p}(\Lambda;\R^2) := \{z\in L^p(\Lambda;\R^2):\norm{z}_{r,p}<\infty\},
\]
where, abusing notation, for $z\in L^p(\Lambda;\R^2)$ we also write $\norm{z}_{r,p}=(\norm{z_1}^p_{r,p}+\norm{z_2}^p_{r,p})^{1/p}$. We see from \eqref{eq:X_tilde} for $\eta<2$, $\varepsilon>0$ and $0\leq t\leq T$ that
\begin{align}
	\norm{\tilde{X}(t)}_{\eta,p}
		& \leq \norm{S(t)X(0)}_{\eta,p} + \int_0^t \norm{S(t-r)F(X(r))}_{\eta,p}dr\nonumber \\
		& \lesssim \norm{X(0)}_{\eta,p} + \int_0^t (t-r)^{-1+\varepsilon/2} \norm{F(X(r))}_{\eta-2+\varepsilon,p}dr.\label{eq:mildBound}
\end{align}
Since $X$ takes values in $C(\Lambda;\R^2)$, it follows easily that the same holds for $F(X(\cdot))$. Choosing $\varepsilon=2-\eta$ and observing that $X(0)\in C^{2+s}(\Lambda;\R^{2})$ for $0<s<1/2$ therefore imply $\sup_{0\leq t\leq T}\norm{\tilde{X}(t)}_{\eta,p} < \infty$. Since this is true for all $p\geq 2$, the Sobolev embedding yields $\tilde{X}\in C([0,T];C^{\eta'}(\Lambda;\R^2))$ for all $\eta'<2$ and so $X\in C([0,T];C^{s'}(\Lambda;\R^2))$ for $0 < s' < 1/2$ by the regularity result for $\bar{X}$ from above,  and therefore (e.g., by \citep[Theorem 3.3.2]{Triebel2010a})
\[
	\sup_{0\leq r\leq T} \norm{F(X(r))}_{s'',p} \lesssim \sup_{0\leq r\leq T} \norm{F(X(r))}_{C^{s'}(\Lambda;\R^2)}<\infty
\]
for all $1/p<s''<s'<1/2$ and $p>2$. Using this in \eqref{eq:mildBound} with $s''=s+\varepsilon<1/2$ for sufficiently small $\varepsilon>0$ and with $2+s$ instead of $\eta$, we conclude at last $\sup_{0\leq t\leq T}\norm{\tilde{X}(t)}_{2+s,p} < \infty$ and thus $\tilde{X}\in C([0,T];C^{2+s}(\Lambda;\R^2))$, finishing the proof.
\end{proof}

\subsection{\label{subsec:Results-on-estimation}Results on parameter estimation}

Since we are not considering Dirichlet boundary conditions, we cannot rely on the Feynman-Kac arguments of  \citep{Altmeyer2019b} and \citep{Altmeyer2020a} to study the action of the semigroup generated by $\Delta$ on $K_{\delta,x_k}$. The following proof avoids this issue, and also holds for more general boundary conditions. The proof is inspired by \citep[Theorem 3]{Altmeyer2020a}, but is fully self-contained.

Consider the decomposition $A=\bar{A}+\tilde{A}$ into linear and nonlinear parts $\bar{A}$ and $\tilde{A}$ according to Section \ref{sec:mild solution}. With this we also set
\begin{align}
\bar{A}_{\delta}^{\Delta}(t,x_{k})&:=\sc{\bar{A}(t)}{\Delta K_{\delta,x_{k}}} = \sigma_A \left\langle \int_0^t S_A(t-s) dW_A(s), \Delta K_{\delta, x_k} \right\rangle  \label{eq:stochIntegral2} \\
&=\sigma_{A}\int_{0}^{t}\sc{S_{A}(t-s)\Delta K_{\delta,x_{k}}}{dW_{A}(s)}, \notag
\end{align}
as well as $\tilde{A}_{\delta}^{\Delta}(t,x_{k})=\sc{\tilde{A}(t)}{\Delta K_{\delta,x_{k}}}$. We also use the linear observed Fisher information $\bar{\c I}_{\delta}=\sum_{k=1}^{M}\int_{0}^{T}\left(\bar{A}_{\delta}^{\Delta}(t,x_{k})\right)^{2}dt$.

\begin{proof}[Proof of Theorem \ref{thm:estimate_DA}]
Using that $\norm{K_{\delta,x_{k}}}=\norm K_{L^{2}(\R)}$, the basic error decomposition \eqref{eq:errorDecomp} can equivalently be written as
\[
	\delta^{-1}(\hat{D}_{A,\delta}-D_{A}) = (\delta^2 \c I_{\delta})^{-1}\delta \c R_{\delta}+\sigma_{A}\norm K_{L^{2}(\R)}(\delta^2\c I_{\delta})^{-1/2}(\c I_{\delta}^{-1/2}\c M_{\delta}).
\]
The martingale part satisfies $\c M_{\delta}=\c M_{\delta}(T)$, where $\c M_{\delta}(t') =\sum_{k=1}^{M}\int_{0}^{t'}A_{\delta}^{\Delta}(t,x_{k})dW_{k}(t)$ is a continuous $\c F_{t'}$-martingale in $t'\geq 0$. Without loss of generality let the $K_{\delta,x_k}$ have disjoint supports, which is true for sufficiently small $\delta$, since $M$ is fixed. But then the processes $W_{k}$
are independent and the quadratic variation of the martingale $\c M_{\delta}(t')$ at $t'=T$ is exactly $\c I_{\delta}$. By a classical time-change \citep[Theorem 3.4.6]{karatzas98} we can write $\c M_{\delta}=\bar{w}_{\c{ I}_{\delta}}$ for a Brownian motion $(\bar{w}(t))_{t\geq 0}$,  possibly defined on an extension of the underlying probability space. We conclude from Proposition \ref{prop:FisherInfo}(i,ii,iii) below that $\c{I}_{\delta}/\E[\c{\bar{I}}_{\delta}]\rightarrow 1$ in probability and
\[
	\frac{\c M_{\delta}}{\c{ I}_{\delta}^{1/2}} =
\frac{\E[\c{ \bar{I}}_{\delta}]^{1/2}}{\c{ I}_{\delta}^{1/2}}\cdot \frac{\bar{w}_{\c{ I}_{\delta}}}{\E[\c{ \bar{I}}_{\delta}]^{1/2}} \r{d} N(0,1).
\]
Proposition \ref{prop:FisherInfo}(i,ii,iii) also shows $\delta^2 \c{I}_{\delta} \rightarrow \kappa$ in probability and the result follows from Slutsky's Lemma and Proposition \ref{prop:FisherInfo}(iv).
\end{proof}

\begin{proposition}
\label{prop:FisherInfo}The following holds as $\delta\rightarrow0$:
\begin{enumerate}
\item $\delta^2 \E[\c{\bar{I}}_{\delta}] \rightarrow \kappa := MT\sigma_{A}^{2}D_{A}^{-1}\norm K_{L^2(\R)}^2\Sigma^{-1}$ with $\Sigma$ from Theorem \ref{thm:estimate_DA},
\item $\c{\bar{I}}_{\delta}/\E[\c{\bar{I}}_{\delta}]\rightarrow 1$ in probability,
\item $\c I_{\delta}=\bar{\c I}_{\delta}+\c O_{\P}(\delta^{-1/2})$,
\item $\c R_{\delta}=\c O_{\P}(\delta^{-1/2})$.
\end{enumerate}
\end{proposition}

\begin{proof}
(i).  We find from \eqref{eq:stochIntegral2} and Itô's isometry (\citep[Proposition 4.28]{DaPrato2014}) that
\begin{equation}
\E[\bar{\c I}_{\delta}]=\sigma_{A}^{2}\sum_{k=1}^{M}\int_{0}^{T}\int_{0}^{t}\norm{S_{A}(s)\Delta K_{\delta,x_{k}}}^{2}dsdt.\label{eq:I_bar}
\end{equation}
The operators $S_{A}(s)$ are self-adjoint such that $\norm{S_{A}(s)\Delta K_{\delta,x_{k}}}^{2}=\sc{S_{A}(2s)\Delta K_{\delta,x_{k}}}{\Delta K_{\delta,x_{k}}}$.
The semigroup identity $2\int_{0}^{t}S_{A}(2s)D_A\Delta K_{\delta,x_{k}}ds=S_{A}(2t)K_{\delta,x_{k}}-K_{\delta,x_{k}}$
therefore implies
\begin{align*}
\E[\bar{\c I}_{\delta}] & =\frac{1}{2}D_{A}^{-1}\sigma_{A}^{2}\sum_{k=1}^{M}\left(\int_{0}^{T}\sc{S_{A}(2t)\Delta K_{\delta,x_{k}}}{K_{\delta,x_{k}}}dt-T\sc{K_{\delta,x_{k}}}{\Delta K_{\delta,x_{k}}}\right)\\
 & =\frac{1}{2}D_{A}^{-1}\sigma_{A}^{2}\sum_{k=1}^{M}\left(\frac{1}{2}D_{A}^{-1}\sc{S_{A}(2T)K_{\delta,x_{k}}-K_{\delta,x_{k}}}{K_{\delta,x_{k}}}-T\sc{K_{\delta,x_{k}}}{\Delta K_{\delta,x_{k}}}\right).
\end{align*}
Noting that $\norm{S_A(T)K_{\delta,x_k}}\leq \norm{K_{\delta,x_k}}$, because the semigroup is contractive, (i) follows from $\sc{K_{\delta,x_{k}}}{\Delta K_{\delta,x_{k}}}=-\delta^{-2}\norm{\frac{\partial}{\partial x}K}_{L^{2}(\R)}^{2}$.

(ii). It is enough to show $\delta^{4}\text{Var}(\bar{\c I}_{\delta})\rightarrow0$,
because this and (i) imply $\text{Var}(\bar{\c I}_{\delta})/\E[\bar{\c I}_{\delta}]^2\rightarrow 0$. Since $M$ is fixed, we can use the Cauchy-Schwarz inequality to obtain the upper bound $\text{Var}(\bar{\c I}_{\delta}) \leq M \sum_{k=1}^{M}\text{Var}(\int_{0}^{T}(\bar{A}_{\delta}^{\Delta}(t,x_{k}))^{2}dt)$. \eqref{eq:stochIntegral2}  shows that the $\bar{A}_{\delta}^{\Delta}(t,x_{k})$ are centered Gaussian random variables. Wick's formula (\citep[Theorem 1.28]{Janson1997})
gives
\begin{align*}
\text{Var}(\bar{\c I}_{\delta}) & \lesssim \sum_{k=1}^{M}\int_{0}^{T}\int_{0}^{T}\text{Cov}\left(\left(\bar{A}_{\delta}^{\Delta}(t,x_{k})\right)^{2},\left(\bar{A}_{\delta}^{\Delta}(t',x_{k})\right)^{2}\right)dt'dt\\
 & =4\sum_{k=1}^{M}\int_{0}^{T}\int_{0}^{t}\text{Cov}\left(\bar{A}_{\delta}^{\Delta}(t,x_{k}),\bar{A}_{\delta}^{\Delta}(t',x_{k})\right)^{2}dt'dt\\
 & =4\sigma_{A}^{4}\sum_{k=1}^{M}\int_{0}^{T}\int_{0}^{t}\left(\int_{0}^{t'}\sc{S_{A}(t-s)\Delta K_{\delta,x_{k}}}{S_{A}(t'-s)\Delta K_{\delta,x_{k}}}ds\right)^{2}dt'dt,
\end{align*}
again using Itô's isometry in the last line. The integrand of the $ds$-integral equals
\begin{align*}
	\norm{S_A((t+t'-2s)/2)\Delta K_{\delta,x_k}}^2 \geq 0.
\end{align*}
On the other hand,
arguing as in (i) by the semigroup identity and $S_{A}(t-s)=S_{A}(t-t')S_{A}(t'-s)$, the $ds$-integral equals
\begin{align*}
 & 0\leq \int_{0}^{t'}\sc{S_{A}(2(t'-s))\Delta K_{\delta,x_{k}}}{S_{A}(t-t')\Delta K_{\delta,x_{k}}}ds  \\
 & \quad =\frac{1}{2D_{A}}\sc{K_{\delta,x_{k}}}{S_{A}(t-t')(-\Delta) K_{\delta,x_{k}}}-\frac{1}{2D_{A}}\sc{S_{A}(2t')K_{\delta,x_{k}}}{S_{A}(t-t')(-\Delta) K_{\delta,x_{k}}}\\
 & \quad = \frac{1}{2D_{A}}\sc{K_{\delta,x_{k}}}{S_{A}(t-t')(-\Delta) K_{\delta,x_{k}}}-\frac{1}{2D_{A}}\sc{(-\Delta)v}{v},
\end{align*}
with $v=S_{A}((t+t')/2)K_{\delta,x_{k}}$.  The operator $-\Delta$ is non-negative, and so $\sc{(-\Delta)v}{v}\geq 0$.  Conclude by the Cauchy-Schwarz inequality
\begin{align*}
\text{Var}(\bar{\c I}_{\delta})
	&  \lesssim \norm{K}^2_{L^2(\R)} \sum_{k=1}^{M}\int_0^T\int_{0}^{t} \norm{S_{A}(t-t')\Delta K_{\delta,x_{k}}}^{2}dt'dt \\
	& = \norm{K}^2_{L^2(\R)} \sum_{k=1}^{M}\int_0^T\int_{0}^{t} \norm{S_{A}(t')\Delta K_{\delta,x_{k}}}^{2}dt'dt \lesssim \E[\bar{\c I}_{\delta}]\lesssim \delta^{-2},
\end{align*}
cf. \eqref{eq:I_bar} and (i), implying $\delta^4\text{Var}(\bar{\c I}_{\delta})\rightarrow 0$ and (ii) follows.

(iii). Recall from Theorem \ref{thm:existenceUniqueness} that $\tilde{A}\in C([0,T];C^{2}(\Lambda))$ $\P$-almost surely. This means
\begin{align}
\left|\tilde{A}_{\delta}^{\Delta}(t,x_{k})\right| & =\left|\sc{\Delta\tilde{A}(t)}{K_{\delta,x_{k}}}\right|\leq\norm{\tilde{A}}_{C([0,T];C^{2}(\Lambda))}\norm{K_{\delta,x_{k}}}_{L^{1}(\Lambda)}=\c O_{\P}(\delta^{1/2}),\label{eq:A_tilde}
\end{align}
using  $\norm{K_{\delta,x_{k}}}_{L^{1}(\Lambda)}\leq\delta^{1/2}\norm K_{L^{1}(\R)}$.
We conclude by the Cauchy-Schwarz inequality and (i), because
\begin{align*}
\left|\c I_{\delta}-\bar{\c I}_{\delta}\right| & \lesssim\sum_{k=1}^{M}\int_{0}^{T}\left(\left(\tilde{A}_{\delta}^{\Delta}(t,x_{k})\right)^{2}+2\left|\tilde{A}_{\delta}^{\Delta}(t,x_{k})\bar{A}_{\delta}^{\Delta}(t,x_{k})\right|\right)dt=\c O_{\P}(\delta+\delta^{1/2}\bar{\c I}_{\delta}^{1/2}).
\end{align*}
(iv). The Cauchy-Schwarz inequality and (i,ii,iii) show
\begin{align*}
|\c R_{\delta}| & \lesssim\c I_{\delta}^{1/2}\left(\sum_{k=1}^{M}\int_{0}^{T}\sc{F_A(X(t))}{K_{\delta,x_{k}}}^{2}dt\right)^{1/2}
	\lesssim \c I_{\delta}^{1/2}\delta^{1/2}\norm{F_{A}(X(\cdot))}_{C([0,T];C(\Lambda))}\norm{K}_{L^{1}(\R)}.
\end{align*}
This is of order $\c O_{\P}(\delta^{-1/2})$, using that $F_{A}(X(\cdot))\in C([0,T];C(\Lambda))$ $\P$-almost surely by Theorem \ref{thm:existenceUniqueness}, recalling that $z\mapsto F_{A}(z)$
is Lipschitz, and the result follows.
\end{proof}


\section{Setup of numerical and real data experiments}\label{sec:setup}

Numerical simulations were performed in the programming language Julia using a finite difference scheme for semilinear SPDEs \citep{Lord2014}. The source code can be obtained from the authors upon request. For comparison to the experimental setup of \citep{Lockley_2015} we let $L=20$, $T\in (0,150]$, and set $dt=T/n$ and $dx=L/m$ as step sizes for time and space discretisations, and choose the number of grid points $n$ and $m$ in time and space such that $dt \asymp (dx)^2$. This ensures that the Courant-Friedrichs-Lewy (CFL) condition is satisfied  \citep{Lord2014} in order to achieve stable simulations. All simulations were performed with parameters and initial conditions from \citep{{Lockley_2015}} obtained by calibrating the deterministic Meinhardt model to the experimental data displayed in \ref{fig:1}(left), which were averaged over 18 different cells. The parameters in \eqref{eq:stochasticMeinhardt}, \eqref{eq:nonlinearities} and \eqref{eq:signal} are taken from \citep[Table S1, Figure 4]{Lockley_2015},
\begin{alignat}{2}
D_{A} &= 4.415 \times 10^{-2}, &\quad D_{I} &= 9.768 \times 10^{-2}, \notag \\
r_{A} &= 2.393 \times 10^{-1}, &\quad r_{I} &= 2.378 \times 10^{-1}, \notag \\
b_{A} &= 2.776 \times 10^{-1}, &\quad b_{I} &= 2.076 \times 10^{-1}, \notag \\
\zeta_{A} &= 5.647 \times 10^{-3}, &\quad \zeta_{I} &= 3.397 \times 10^{-1}, \notag \\
a &= 1.280 \times 10^{-2}. &\quad & \notag
\end{alignat}
The initial conditions for the activator $A$ and inhibitor $I$ are taken correspondingly from \citep[Table S2, Figure 4]{Lockley_2015}.
By adding the stochastic part in \eqref{eq:stochasticMeinhardt} it may happen that the concentrations of $A$ or $I$ become negative. Such realisations were not taken into account when computing the time to repolarisation from simulated data.  

For the real data analysis in Section \ref{sec:realData} data for the 18 single cells were used. They each contain $M=100$ spatial measurements for evolving over time. For a detailed description of the original experimental data see \citep{dalous2008reversal} and \citep{lockley2017image}.

\bibliographystyle{apalike}
\bibliography{refs}

\end{document}